\documentclass{article}
\usepackage[utf8]{inputenc}
\usepackage{amsmath,amssymb,amsthm}
\usepackage[authoryear,square]{natbib}
\usepackage{authblk}

\title{Decomposing Real Square Matrices via Unitary Diagonalization}

\author{Théo Trouillon\thanks{Xerox Research Centre Europe, 
Université Grenoble Alpes, \texttt{theo.trouillon@xrce.xerox.com}},
Christopher R. Dance\thanks{Xerox Research Centre Europe, \texttt{chris.dance@xrce.xerox.com}},
Éric Gaussier\thanks{Université Grenoble Alpes, \texttt{eric.gaussier@imag.fr}} and\\
Guillaume Bouchard\thanks{University College London, \texttt{g.bouchard@cs.ucl.ac.uk}} 
}

\date{}

\renewcommand{\cite}{\citep}

\newcommand{\complexSpace}{\mathbb{C}}
\renewcommand{\Re}{\mathbb{R}}
\newcommand{\C}{\complexSpace} 
\newcommand{\R}{\Re} 

\newcommand{\Mnr}{\R^{n \times n}}
\newcommand{\Mnc}{\C^{n \times n}}

\newcommand{\real}{\mathrm{Re}}
\newcommand{\imag}{\mathrm{Im}}
\newcommand{\sign}{\mathrm{sign}}
\newcommand{\rank}{\mathrm{rank}}
\newcommand{\srank}{\mathrm{rank}_{\pm}}
\newcommand{\T}{^\top}

\newtheorem{thm}{Theorem}
\newtheorem{defn}{Definition}
\newtheorem{cor}{Corollary}

\begin{document}
\maketitle

\begin{abstract}
Diagonalization, or eigenvalue decomposition, is very useful in many areas of applied mathematics, including signal processing and quantum physics. Matrix decomposition is also a useful tool for approximating matrices as the product of a matrix and its transpose, which relates to unitary diagonalization. As stated by the spectral theorem, \emph{only} normal matrices are unitarily diagonalizable. However we show that \emph{all} real square matrices are the real part of some unitarily diagonalizable matrix. 
\end{abstract}

\section{Introduction}

Matrix decomposition is a very well studied field of linear algebra,
in particular, eigenvalue decomposition has applications in many scientific areas.
A specific case arises when the eigenvectors form an orthonormal
basis. In other words, the matrix formed by the column eigenvectors is orthogonal. In this case we speak of orthogonal diagonalization.

\begin{defn}
A real square matrix $A \in \Mnr$ is orthogonally diagonalizable if it can be
written $A= S \Lambda S\T$, with $S, \Lambda \in \Mnr$, $\Lambda$ diagonal, and $S$ orthogonal: $SS\T = S\T S = I$.
\end{defn}

Decompositions of $n$-by-$n$ 
real matrices $A$ that can be expressed as $A = S \Lambda S\T$  
are especially relevant when rows and columns 
of a square matrix represent the same objects of some 
underlying problem. Examples of motivating applications include 
spectral analysis of graphs \cite{cvetkovic1997eigenspaces}, 
and decomposition of covariance matrices such as principal 
component analysis \cite{Jolliffe:1986}.

The spectral theorem tells us that a matrix 
is orthogonally diagonalizable if and only if it is symmetric.
However in many cases asymmetric square matrices arise,
and yet it would still be useful to have a symmetric 
decomposition of these matrices. This is frequently the case in
machine learning problems, for example: 
learning vectorial representations of words from co-occurence matrices
\cite{pennington2014glove}, and of entities from knowledge graphs
\cite{nickel2015review}. 
In these applications, row $i$ of the matrix represent the
same object as column $i$. Therefore it would be natural to have unique 
representations, or embeddings, $s_i$ for each object. 
But representing element $A_{ij}$ of the matrix by $s_i\T \Lambda s_j$ would yield
a symmetric matrix $A$.
To avoid this problem, 
practitioners either use sophisticated transformations of the original 
matrix to symmetrize it, or they consider more complicated decompositions,
or they simply use non-symmetric decompositions like singular value decomposition.


We introduce a new decomposition using unitary
diagonalization, the generalization of orthogonal diagonalization
to the complex field. This allows decomposition of \emph{arbitrary} real square matrices
with unique representations of rows and columns. 
Let us first recall some notions of complex linear algebra.

Let $\overline{A} \in \C^{m \times n}$ denote the complex conjugate of the complex 
matrix $A $. We shall write $A^* \in \C^{n \times m}$ the conjugate-transpose $A^*= (\overline{A})\T 
= \overline{A\T}$. The conjugate transpose is also often written $A^\dagger$.

\begin{defn}
A complex square matrix $A \in \Mnc$ is unitarily diagonalizable if it can be
written as $A= S \Lambda S^*$, with $S, \Lambda \in \Mnc$, $\Lambda$ diagonal, and $S$ unitary: $SS^* = S^*S = I$.
\end{defn}

\begin{defn}
A complex square matrix $A$ is normal if it commutes with its
conjugate-transpose $AA^* = A^*A$.
\end{defn}



We can now state the less well-known version of the spectral theorem that applies to the complex domain.

\begin{thm}[Spectral Theorem {{\cite{axler1997linear}}}]
\label{spectral_thm}

Suppose $A$ a complex square matrix. Then A is unitarily diagonalizable if and only if
A is normal.

\end{thm}

Among normal matrices, there are some purely real matrices that are not
symmetric, such as the skew-symmetric matrices, and thus can be diagonalized
in the complex domain. As we only focus on \emph{real} square matrices in this work, let us
summarize all the cases where $A$ is real square and $A= S \Lambda S^*$, with $S,\Lambda \in \Mnc$, $\Lambda$ is diagonal and $S$ is unitary:

\begin{itemize}
    \item $A$ is symmetric if and only if $A$ is orthogonally diagonalizable,
    with $S$ and $\Lambda$ are purely real.
    \item $A$ is normal and non-symmetric if and only if A is unitarily 
    diagonalizable and $S$ and $\Lambda$  are \emph{not} both purely real.
    
    \item $A$ is not normal if and only if $A$ is not unitarily diagonalizable.
\end{itemize}

In the following, we generalize all three cases by showing that, for any 
$A \in \Mnr$, there exists a unitary diagonalization in the complex domain,
of which the real part equals $A$:

$$ A = \real(S \Lambda S^* ).$$

In other words, the unitary diagonalization is projected on the real subspace.

\section{Real Square Matrices}

\begin{thm}
\label{main_thm}

Suppose $A \in \Mnr$ is a real square matrix. Then there exists
a normal matrix $X \in \Mnc$ such that $ \real(X) = A$.

\end{thm}

\begin{proof}

Let form the complex square matrix $X \triangleq A + iA\T$, and derive $X^*$:

\begin{eqnarray*}
X^* = A\T - iA = -i(iA\T + A) = -iX .
\end{eqnarray*}

Therefore $X$ is normal:
\begin{eqnarray*}
XX^* = X(-iX) = (-iX)X = X^*X.
\end{eqnarray*}

\end{proof}

Remark that there also exists a normal matrix $X = A\T + iA$ such that $ \imag(X) = A$.


Following Theorem~\ref{spectral_thm} and Theorem~\ref{main_thm}, any
real square matrix can be written as the real part of a diagonal matrix through
a unitary change of basis.

\begin{cor}
\label{cor_real_diag}

Suppose $A \in \Mnr$ is a real square matrix. Then there exist $S,\Lambda \in \Mnc$, where $S$ is unitary, and $\Lambda$ diagonal, such that $A = \real(S \Lambda S^* )$.

\end{cor}

\begin{proof}
From Theorem \ref{main_thm}, we can write $A =\real(X) $, where $X$ is a normal matrix,
and from Theorem \ref{spectral_thm}, $X$ is unitarily diagonalizable.
\end{proof}

Let us discuss the rank of such diagonalizations.
First, we recall the definition of the rank of a matrix.

\begin{defn}
$\rank(A)=k$ with $A$ an $m$-by-$n$ matrix over an arbitrary field $F$, if 
$A$ has exactly $k$ linearly independent columns.
\end{defn}

Also note that if $A$ is diagonalizable $A = S \Lambda S^{-1}$ with $\rank(A)=k$, $\Lambda$ has $k$ non-zero diagonal entries for some $\Lambda$ diagonal and some matrix $S$ invertible.
From this it is easy to derive a known additive property of the rank:
$\rank(B+C) \leq \rank(B) + \rank(C)$, with $B,C$ $m$-by-$n$ matrices over an arbitrary
field $F$.

We now show that any rank-$k$ real square matrix can be reconstructed 
from a $2k$-dimensional unitary diagonalization.

\begin{cor}

\label{cor_real_rank}

Suppose $A \in \Mnr$ arbitrary, and $rank(A) = k$. Then there exist 
$S \in \C^{n \times 2k}$, $\Lambda \in \C^{2k \times 2k}$, where 
the columns of $S$ form an orthonormal basis of $\C^{2k}$, and $\Lambda$ diagonal, such that $A = \real(S \Lambda S^* )$.
\end{cor}

\begin{proof}
Consider the complex square matrix $X \triangleq A + iA\T$. We have $\rank(iA\T) = \rank(A\T) = \rank(A) = k$.

Thus $\rank(X) \leq \rank(A) + \rank(iA\T) = 2k$.


The proof of Theorem \ref{main_thm} shows that $X$ is normal, thus: $X = S \Lambda S^* $ with $S \in \C^{n \times 2k}$, $\Lambda \in \C^{2k \times 2k}$ where 
the columns of $S$ form an orthonormal basis of $\C^{2k}$, and $\Lambda$ diagonal.

\end{proof}

Given that such decomposition always exists in dimension $n$ (Theorem \ref{main_thm}),
this upper bound is not relevant when $rank(A) \geq \frac{n}{2}$.

\section{Low sign-rank Matrices}

We interest ourselves here in square sign matrices,
$Y \in \{-1,1\}^{n \times n}$, and how they can be reconstructed using the sign function of a real matrix $X\in \Mnr$ : $Y = \sign(X)$, where $\sign(X)_{ij} = \sign(x_{ij})$. There are many such matrices $X$. More precisely, all such matrices $X$ constitute a different orthant of $\Mnr$ for each sign matrix $Y$.

The use of the sign function is here not arbitrary, as it maps to a known
complexity measure for sign matrices:  the sign-rank \cite{linial2007complexity}.

\begin{defn}
The sign-rank $\srank(Y)$  of an $m$-by-$n$ sign matrix $Y$, is the smallest rank 
among all the
$m$-by-$n$ real matrices that have the same sign-pattern as $Y$.
Formally: 
$$\srank(Y) = \min_{A\in \R^{m \times n}} \{\rank(A) | \sign(A) = Y \}.$$
\end{defn}

Using Corollary \ref{cor_real_rank}, we can now show that any square sign matrix of sign-rank $k$ can be reconstructed from a rank-$2k$ unitary diagonalization.

\begin{cor}
\label{cor_sign_rank}
Suppose $Y \in \{-1,1\}^{n \times n}$, $\srank(Y)=k$. Then there exists 
$S \in \C^{n \times 2k}$, $\Lambda \in \C^{2k \times 2k}$ where 
the columns of $S$ form an orthonormal basis of $\C^{2k}$, and $\Lambda$ diagonal,
such that $Y = \sign(\real(S \Lambda S^* ))$.
\end{cor}

\begin{proof}
By definition, if $\srank(Y)=k$, there exists a real square matrix $A$ such that
$\rank(A)=k$ and $\sign(A)=Y$.
From Corollary \ref{cor_real_rank}, $A = \real(S \Lambda S^* )$ where $S \in \C^{n \times 2k}$, $\Lambda \in \C^{2k \times 2k}$ where 
the columns of $S$ form an orthonormal basis of $\C^{2k}$, and $\Lambda$ diagonal.
\end{proof}

To the best of our knowledge, previous attempts to approximate 
the sign-rank did not use complex
numbers. They showed the existence of compact factorizations under 
conditions on the sign matrix~\cite{nickel2014reducing}, or only in specific cases~\cite{bouchard2015approximate}.
This results shows that if a sign matrix
has sign-rank $k$, then it can be reconstructed by a $2k$-dimensional
unitary diagonalization in the complex space.

\paragraph{Example\\}

Consider the following $2$-by-$2$ sign matrix:
\[
Y=
  \begin{bmatrix}
    -1 & -1\\
     \phantom{-}1 & \phantom{-}1\\
  \end{bmatrix}
\]

Not only is this matrix not normal, but one can also easily check that there is no \emph{real} normal $2$-by-$2$ matrix that has the same sign-pattern as $Y$.
Clearly, $Y$ is a rank-$1$ matrix since its columns are linearly dependent,
hence its sign-rank is also $1$. From Corollary \ref{cor_sign_rank}, we know that there is a normal matrix of which the real part has the same sign-pattern as $Y$,
and of which the rank is less than or equal to $2$. 


However, there is no such rank-$1$ unitary diagonalization of $Y$.
Otherwise we could find a 2-by-2 complex matrix $X$
such that $\real(x_{11}) < 0$ and $\real(x_{22}) > 0$, 
where $x_{11} = s_1 \lambda \bar{s}_1 = \lambda |s_1|^2$, $x_{22} = s_2 \lambda \bar{s}_2 
= \lambda |s_2|^2$,
$s \in \C^2, \lambda \in \C$. This is obviously unsatisfiable.
This example generalizes to any $n$-by-$n$ square sign matrix that only
has $-1$ on its first row and is hence rank-1, the same argument holds
considering $\real(x_{11}) < 0$ and $\real(x_{nn}) > 0$.

This shows that the upper bound
on the rank of the unitary diagonalization showed in Corollaries \ref{cor_real_rank} 
and \ref{cor_sign_rank} is strictly greater than $k$, the rank or sign-rank, 
of the decomposed matrix.

Though this decomposition is clearly not unique,
it can capture meaningful representations of the input matrix in a low-rank
decomposition setting. \citet{trouillon2016} show state-of-the-art results
by jointly decomposing a set of square sign matrices representing relational 
data.








\clearpage

\bibliography{linear_algebra_refs}
\bibliographystyle{plainnat}

\end{document}